\documentclass[10pt,a4paper]{article}

\usepackage[T1]{fontenc}

\usepackage[latin1]{inputenc}						
\usepackage{amsmath}

\usepackage{xcolor, setspace}
\definecolor{bl}{rgb}{0.0,0.2,0.6}

\usepackage{sectsty}
\usepackage[compact]{titlesec}
\usepackage{tabularx}

\usepackage{amsthm, amsmath}
\usepackage{enumerate}

\usepackage{amssymb, mathrsfs, amscd}
\usepackage[all]{xy}

\usepackage{graphics}
\usepackage{graphicx}
\usepackage{enumerate}

\newtheorem{theorem}{Theorem}[section]

\newtheorem{lemma}[theorem]{Lemma}

\newtheorem{proposition}[theorem]{Proposition}
\newtheorem{corollary}[theorem]{Corollary}
\newtheorem*{corollary*}{Corollary}

\newtheorem{remark}[theorem]{Remark}

\mathchardef\mhyph="2D

\title{On dimensions of groups with cocompact classifying spaces for
  proper actions} 

\author{Ian J. Leary \and Nansen Petrosyan} 
\newcommand{\ebar}{\underline{E}}

\newcommand{\vcd}{\hbox{\rm vcd}}
\newcommand{\cdbar}{\underline{\hbox{\rm cd}}}
\newcommand{\gdbar}{\underline{\hbox{\rm gd}}}
\newcommand{\mF}{\mathcal {F}}

\newcommand{\Z}{\mathbb Z}
\newcommand{\F}{\mathbb F}
\newcommand{\G}{\Gamma}
\newcommand{\orb}{\mathcal{O}_{\mF}G}
\newcommand{\orbe}{\mathcal{O}_{\{e\}}G}
\newcommand{\orbmod}{\mbox{Mod-}\mathcal{O}_{\mF}G}

\newcommand{\Hom}{\mathrm{Hom}}

\newcommand{\rH}{\mathrm{H}}

\newcommand{\sing}{\mathrm{sing}}
\newcommand{\semi}{\rtimes} 

\newcommand{\ev}{\mathrm{ev}}
\newcommand{\mG}{\mathcal{G}}
\newcommand{\mW}{\mathcal{W}}

\medskip

\begin{document}

\maketitle
\begin{abstract}  
We construct groups $G$ that are virtually torsion-free and have
virtual cohomological dimension strictly less than the minimal
dimension for any model for $\ebar G$, the classifying space for
proper actions of $G$.  They are the first examples that have these
properties and also admit cocompact models for $\ebar G$.  We exhibit
groups $G$ whose virtual cohomological dimension and Bredon
cohomological dimension are two that do not admit any 2-dimensional
contractible proper $G$-CW-complex.
\end{abstract}

\section{Introduction} 
If $G$ is a virtually torsion-free group, the virtual cohomological 
dimension $\vcd G$, is defined to be the cohomological dimension of 
a torsion-free finite-index subgroup $H\leq G$; a lemma due to 
Serre shows that this is well defined~\cite[VIII.3.1]{brownco}.  
Now suppose that $X$ is a contractible $G$-CW-complex that is 
{\em proper}, in the sense that all cell stabilizers are finite. 
In this case any torsion-free subgroup $H$ will act freely on $X$
and so $X/H$ is a classifying space or Eilenberg-Mac~Lane space
$BH$ for $H$.  In particular, $\vcd G$ provides a lower bound 
for the dimension of any such $X$.  K.~S.~Brown asked whether 
this lower bound is always attained~\cite[ch.~2]{brownwall} or 
\cite[VIII.11]{brownco}:  \\

\noindent {\bf Brown's Question ({\normalfont Weak Form}).} 
Does every virtually torsion-free group $G$ admit a contractible proper 
$G$-CW-complex of dimension $\vcd G$? \\

\noindent Until now, this form of Brown's question has remained unanswered.  
We give examples of groups $G$ with $\vcd G=2$ that do not admit
any 2-dimensional contractible proper $G$-CW-complex in Theorem~\ref{acycthm} 
below.  

One reason why this question has been so elusive is that there 
are many different equivariant homotopy types of contractible
proper $G$-CW-complexes.  The most natural example is the 
classifying space for proper $G$-actions, $\underline{E}G$, 
which plays the same role in the homotopy category of proper 
$G$-CW-complexes as $EG$ plays for free $G$-CW-complexes.  
A model for $\ebar G$ is a proper $G$-CW complex $X$ 
such that for any finite $F\leq G$, the $F$-fixed point set 
$X^F$ is contractible.  Such an $X$ always exists, and is unique 
up to equivariant homotopy equivalence.  Let $\gdbar G$ denote the
minimal dimension of any model for $\ebar G$.  

The version of Brown's question that concerns
$\underline{E}G$~\cite[ch.~2]{brownwall} or 
\cite[VIII.11]{brownco} is usually asked in the form: \\

\noindent {\bf Brown's Question ({\normalfont Strong
   Form}).} 
Does $\gdbar G=\vcd G$ for every virtually 
torsion-free $G$? \\
 
\noindent We prefer to split this question into two separate
questions.  There is an algebraic dimension 
$\cdbar G$ that bears a close relationship to $\gdbar G$, analogous
to the relationship between cohomological dimension and the 
minimal dimension of an Eilenberg-Mac~Lane space.  It can be 
shown that $\cdbar G= \gdbar G$ except that there may exist 
$G$ for which $\cdbar G=2$ and $\gdbar G=3$, and $\cdbar G$ 
is an upper bound for the cohomological dimension of any 
torsion-free subgroup of $G$~\cite{LuckMeintrup}.  In view 
of this we may split the strong form of Brown's question into
two parts, one geometric and one algebraic.
\begin{itemize} 
\item[]
Does there exist $G$ for which $\gdbar G\neq \cdbar G$? 

\item[]
Does there exist virtually torsion-free
$G$ for which $\cdbar G> \vcd G$?  
\end{itemize}
Examples of virtually torsion-free groups $G$ for which $\gdbar G=3$
and $\cdbar G=2$ were given in~\cite{BLN}.  These groups $G$ are
Coxeter groups.  Examples of $G$ for which $\cdbar G>\vcd G$ were
given in~\cite{VFG}, and more recently in~\cite{CMP,DDNP}.  The
advantage of the examples in~\cite{CMP,DDNP} is that in some sense
they have the least possible torsion.  For any virtually torsion-free
$G$, it can be shown that $\cdbar G$ is bounded by the
sum $\vcd G + \ell(G)$, where $\ell(G)$ is the maximal length of a
chain of non-trivial finite subgroups of $G$~\cite[6.4]{Luck1}.  
This bound is attained for the examples in~\cite{CMP,DDNP} but not for
the examples in~\cite{VFG}.
To date, all constructions of groups $G$ for which $\cdbar G>\vcd G$
have used finite extensions of Bestvina-Brady groups~\cite{BB}, and 
none of these groups $G$ admit a cocompact model for $\ebar G$.  One 
of our main results is the construction of virtually torsion-free $G$ 
admitting a cocompact $\ebar G$ for which $\cdbar G>\vcd G$.  
Amongst our examples, the easiest to describe are extensions of 
a right-angled Coxeter group by a cyclic group of prime order.  
By taking instead a cyclic extension of a torsion-free finite
index subgroup of the same Coxeter group we obtain examples 
with cocompact $\ebar G$ and for which $\cdbar G = \vcd G+ \ell(G)$.

We say that a simplicial action of a group on a simplicial complex is
{\it admissible} if the setwise stabilizer of each simplex equals to
its pointwise stabilizer.  Many of the other terms used in the
statements of our main theorems will be defined below.  

\begin{theorem}\label{th:ksbq}Let $L$ be a finite $n$-dimensional 
acyclic flag complex with an admissible simplicial action of a finite
group $Q$, and let $W_L$ be the corresponding right-angled Coxeter
group so that $Q$ acts as a group of automorphisms of $W_L$.  Let $N$
be any finite-index normal subgroup of $W_L$ such that $N$ is
normalized by $Q$, and let $G$ be the semidirect product $N\semi Q$. 
This $G$ admits a cocompact model for $\ebar G$.  Let $L^\sing$ denote
the subcomplex of $L$ consisting of points with non-trivial stabilizer
in $Q$.
\[\hbox{If}\,\,\, H^n(L,L^\sing)\neq 0,\,\,\,\hbox{then}\,\,\,
\cdbar G=n+1 \,\,\,\hbox{and}\,\,\, \vcd G\leq n.\]

Now suppose that $L_i$, $Q_i$, $n_i$, $N_i$, $W_i$ and $G_i$ are 
as above for $i=1,\ldots,m$, and let $\Gamma :=G_1\times \cdots 
\times G_m$.  As before, there is a cocompact model for $\ebar\Gamma$, 
and   
\[\hbox{if}\,\,\, \bigotimes_{i=1}^m H^{n_i}(L_i,L_i^\sing) \neq 0, 
\,\,\,\hbox{then}\,\,\, \cdbar \Gamma = 
m+\sum_{i=1}^m n_i \,\,\,\hbox{and}\,\,\,
\vcd \Gamma\leq \sum_{i=1}^mn_i.\] 

Furthermore $\vcd G=n$ if either $L$ is a barycentric subdivision or
$L^\sing$ is a full subcomplex of $L$.  Similarly, $\vcd \Gamma=\sum_i
n_i$ if for each $i$, either $L_i$ is a barycentric subdivision or 
$L_i^\sing$ is a full subcomplex of $L_i$.  
\end{theorem} 

\begin{corollary}\label{cora} For each $m\geq 1$ there exists a
virtually torsion-free group $\Gamma_m$ admitting a cocompact 
$\ebar \Gamma_m$ and such that $$\cdbar \Gamma_m = 3m 
> \vcd \Gamma_m =2m.$$ 

For each $m\geq 1$ there exists a virtually torsion-free 
group $\Lambda_m$ admitting a cocompact $\ebar \Lambda_m$ and such 
that $$\cdbar \Lambda_m = 4m = 
\vcd \Lambda_m +\ell(\Lambda_m) > \vcd\Lambda_m = 3m.$$ 
Furthermore, $\Lambda_m$ may be chosen so that either every finite subgroup is 
cyclic or, for any fixed prime $q$, every nontrivial finite subgroup is 
 abelian of exponent $q$.  
\end{corollary} 

In contrast to the above results, Degrijse and Mart\'\i{n}ez-P\'erez 
have shown that $\vcd G=\cdbar G$ for a large class of groups that 
contains all (finitely generated) Coxeter groups~\cite{DDMP}.  

\begin{theorem}\label{acycthm} 
Suppose that $L$ is a finite 2-dimensional acyclic flag complex 
such that the fundamental group of $L$ admits a non-trivial 
unitary representation $\rho:\pi_1(L)\rightarrow U(n)$ for 
some $n$.  Then $\vcd W_L=\cdbar W_L=2$, 
there is a cocompact 3-dimensional model for $\ebar W_L$, and 
yet there exists no proper 2-dimensional contractible $W_L$-CW-complex.  
\end{theorem} 

Theorem~\ref{acycthm} strengthens a result from~\cite{BLN}, and
gives the first negative answer to the weak form
of Brown's question.  A different argument was used in~\cite{BLN} to
show that $\cdbar W_L=2 < \gdbar W_L=3$ for some of the flag complexes
$L$ that appear in Theorem~\ref{acycthm}.  

We remark that finitely generated Coxeter groups are linear over
$\Z$~\cite{brown}, and that this property passes to subgroups and to
finite extensions.  Hence all of the groups appearing in the above
statements are linear over $\Z$.  As will be seen from the proofs,
each group appearing in our statements acts properly and cocompactly
on a CAT(0) cube complex; in particular they are all CAT(0) groups.  A
right-angled Coxeter group $W_L$ is (Gromov) hyperbolic if and only if
$L$ satisfies the flag-no-square condition.  Hyperbolicity passes to
finite index subgroups and finite extensions.  Since any 2- or
3-dimensional flag complex admits a flag-no-square
subdivision~\cite{dran,prsw} it follows that the groups $\Gamma_1$ and
$\Lambda_1$ in Corollary~\ref{cora} and the groups $W_L$ in
Theorem~\ref{acycthm} may be taken to be hyperbolic (and CAT($-1$), a
possibly stronger property) in addition to their other stated
properties.

\section{Classifying spaces and Bredon cohomology}

The algebraic analogs of the geometric finiteness properties exhibited
by classifying spaces of groups for families of subgroups are
formulated using Bredon cohomology.  This cohomology theory was
introduced by Bredon in \cite{Bredon} for finite groups and was
generalised to arbitrary groups by L\"{u}ck (see \cite{Luck}).

Let $G$ be a discrete group.  A family $\mathcal{F}$ of subgroups 
of $G$ is a non-empty set of subgroups which is closed under 
conjugation and taking subgroups, in the sense that if $H\in \mathcal{F}$,
$g\in G$ and $K\leq H$, then $K\in\mathcal{F}$ and $gHg^{-1}\in 
\mathcal{F}$.  The \emph{orbit category} $\orb$ is 
the category with objects the left cosets $G/H$ for all $H \in
\mathcal{F}$ and morphisms all $G$-equivariant functions between
the objects. In $\orb$, every morphism $\varphi: G/H \rightarrow G/K$
is completely determined by $\varphi(H)$, since
$\varphi(xH)=x\varphi(H)$ for all $x \in G$. Moreover, there exists a
morphism \[G/H \rightarrow G/K : H \mapsto xK\] if and only if
$x^{\scriptscriptstyle -1}Hx \subseteq K$.

An \emph{$\orb$-module} is a contravariant functor $M: \orb
\rightarrow \mathbb{Z}\mbox{-mod}$. The \emph{category of
$\orb$-modules} is denoted by $\orbmod$. By definition, it has as objects
all $\orb$-modules and as morphisms all natural
transformations between these objects.  The category $\orbmod$ is an 
abelian category that contains enough projectives and 
so one can construct bi-functors
$\mathrm{Ext}^{n}_{\orb}(-,-)$ that have all the usual properties. The
\emph{$n$-th Bredon cohomology of $G$} with coefficients $M \in
\orbmod$ is by definition
\[ \mathrm{H}^n_{\mathcal{F}}(G;M)=
\mathrm{Ext}^{n}_{\orb}(\underline{\mathbb{Z}},M), \] where
$\underline{\mathbb{Z}}$ is the constant functor, which sends each
object to $\Z$ and each morphism to the identity map on $\Z$.  There
is also a notion of \emph{Bredon cohomological dimension} of $G$ for
the family $\mathcal{F}$, denoted by $\mathrm{cd}_{\mathcal{F}}(G)$
and defined by
\[ \mathrm{cd}_{\mathcal{F}}(G) = \sup\{ n \in \mathbb{N} \ | \
\exists M \in \orbmod :  \mathrm{H}^n_{\mathcal{F}}(G; M)\neq 0 \}. \]

When $\mF$ is the family of finite subgroups, then
$\rH^*_{\mF}(G,M)$ and $\mathrm{cd}_{\mF}G$ are
denoted by $\underline{\rH}^*(G,M)$ and $\underline{\mathrm{cd}}G$,
respectively.  Since the augmented cellular chain complex of any model
for $E_{\mF}G$ yields a projective resolution of
$\underline{\Z}$ that can be used to compute
$\mathrm{H}_{\mF}^{\ast}(G;-)$, it follows that
$\mathrm{cd}_{\mathcal{F}}(G) \leq
\mathrm{gd}_{\mathcal{F}}(G)$. Moreover, it is known (see for 
example~\cite[0.1]{LuckMeintrup}) 
that \[\mathrm{cd}_{\mathcal{F}}(G) \leq
\mathrm{gd}_{\mathcal{F}}(G) \leq \max\{3,
\mathrm{cd}_{\mathcal{F}}(G) \}.\]

For any $\Z G$-module $M$, one may define an $\orb$-module $\underline{M}$ 
by $$\underline{M}(G/H)=\Hom_G(\Z [G/H],M);$$ note that this is compatible with 
the notation $\underline{\Z}$ introduced earlier and that this functor is isomorphic to the fixed-point functor
$$M^{\overline{\;\;}}:\orb\to \mathbb{Z}\mbox{-mod}: G/H\mapsto M^H.$$  For any $G$-CW-complex
$X$ with stabilizers in $\mF$, it can be shown that Bredon cohomology 
with coefficients in $\underline{M}$ is naturally isomorphic to the 
ordinary equivariant cohomology of $X$ with coefficients in $M$: 
$H^*_\mF(X;\underline{M})\cong H^*_G(X;M)$.  This follows because 
the adjointness of the restriction and coinduction functors between
$\Z G$-mod and $\orbmod$ associated to the functor $$G=\orbe\to \orb
:G/{\{e\}}\mapsto G/{\{e\}}$$ gives an isomorphism of  cochain
complexes  
\[\Hom_\mF(C^\mF_*(X),\underline{M})\cong \Hom_G(C_*(X), M).\] 

A {\it subfamily} of a family $\mF$ of subgroups of $G$ is another 
family $\mG\subseteq \mF$.  For a subfamily $\mG$ and a $G$-CW-complex 
$X$ with stabilizers in $\mF$, the $\mG$-singular set $X^{\mG\mhyph\sing}$ 
is the subcomplex consisting of points of $X$ whose stabilizer is 
not contained in $\mG$.  When $\mG$ consists of just the trivial 
subgroup this is the usual singular set and we write $X^\sing$ for 
$X^{\mG\mhyph\sing}$.  

Given a $\Z G$-module $M$ and a subfamily $\mG$ of $\mF$, we 
define two further $\orb$-modules: a submodule $\underline{M}_{\leq\mG}$ of 
$\underline{M}$, and the corresponding quotient module 
$\underline{M}_{> \mG}$.  These are defined by 
\[ \underline{M}_{>\mG}: 
 G /H \mapsto  
\left\{\begin{array}{lll} \Hom_G(\Z [G/H],M) & \mbox{if}&  H\notin\mG,\\
                    0            & \mbox{if}&  H\in\mG, \end{array}
\right.  \] 
\[ \underline{M}_{\leq\mG}: 
 G /H \mapsto  
\left\{\begin{array}{lll} 0 & \mbox{if}&  H\notin\mG,\\
                    \Hom_G(\Z [G/H],M)  & \mbox{if}&  H\in\mG. \end{array}
\right. \] 
By construction there is a short exact sequence of $\orb$-modules 
\[\underline{M}_{\leq\mG} \rightarrowtail \underline{M} 
\twoheadrightarrow \underline{M}_{> \mG}.\] 
Hence, there is a 
short exact sequence of cochain complexes  
\[0\to \Hom_\mF(C^\mF_*(X),\underline{M}_{\leq\mG})\to 
\Hom_\mF(C^\mF_*(X),\underline{M})\to 
\Hom_\mF(C^\mF_*(X),\underline{M}_{> \mG})\to 0\]
which gives rise to a long exact sequence in Bredon cohomology.  
By construction $$\Hom_\mF(C^\mF_*(X),\underline{M}_{> \mG})\cong
\Hom_\mF(C^\mF_*(X^{\mG\mhyph\sing}),\underline{M}),$$ 
and by adjointness isomorphism noted
earlier $$\Hom_\mF(C^\mF_*(X^{\mG\mhyph\sing}), \underline{M})\cong
\Hom_G(C_*(X^{\mG\mhyph\sing}), M).$$ 
It follows that there is a natural identification between Bredon 
cohomology with coefficients in $\underline{M}_{>\mG}$ and 
the equivariant cohomology of $X^{\mG\mhyph\sing}$ with 
coefficients in $M$: 
\[H_\mF^*(X;\underline{M}_{>\mG})\cong H_G^*(X^{\mG\mhyph\sing};M),\] 
we deduce that the Bredon cohomology 
with coefficients in $\underline{M}_{\leq \mG}$ is isomorphic 
to the equivariant cohomology of the pair $(X,X^{\mG\mhyph\sing})$ 
with coefficients in $M$: 
\[H_\mF^*(X;\underline{M}_{\leq\mG}) \cong H_G^*(X,X^{\mG\mhyph\sing};M).\]  
Hence we obtain the following.  

\begin{proposition} \label{bredon_prop}
Let $\mF$ be a family of subgroups of $G$, with $\mG$ a subfamily, 
let $X$ be any model for $E_\mF G$, and let $H$ be a finite-index
subgroup of $G$.  There exists an $\orb$-module module $\mathcal{C}$ 
such that the Bredon cohomology of the group $G$ with coefficients
in $\mathcal{C}$ computes the ordinary cohomology of the pair 
$(X/H,X^{\mG\mhyph\sing}/H)$: 
\[H_\mF^*(G;\mathcal{C}) \cong H^n(X/H,X^{\mG\mhyph\sing}/H;\Z).\] 
Furthermore, each abelian group $\mathcal{C}(G/K)$ is finitely 
generated.  
\end{proposition} 

\begin{proof} 
Let $M$ be the permutation module $\Z [G/H]$, and let 
$\mathcal{C}:=\underline{M}_{\leq \mG}$.  Then
\begin{align*}H_\mF^*(G;\mathcal{C})&\cong H_\mF^*(X;\mathcal{C}) 
\cong H_G^*(X,X^{\mG\mhyph\sing};\Z [G/H]) \\
&\cong 
H_H^*(X,X^{\mG\mhyph\sing};\Z)\cong 
H^*(X/H,X^{\mG\mhyph\sing}/H;\Z),
\end{align*} 
where the first two isomorphisms follow from the discussion above
and the third because $H$ has finite index in $G$.  
\end{proof} 

\section{Right-angled Coxeter groups}

In this section we describe the results that we require concerning
right-angled Coxeter groups and the Davis complex.  The material up
to and including Corollary~\ref{corequiv} is standard; for more
details we refer the reader to \cite{davis2}~or~\cite{davis1}.

A {\it right-angled Coxeter system} consists of a group $W$ called a
{\it right-angled Coxeter group} together with a set $S$ of
involutions that generate $W$, subject to only the relations that
certain pairs of the generators commute.  We will always assume that
$S$ is finite.  The defining relators have the forms $s^2=1$ and
$stst=1$ where $s\in S$ and $t$ ranges over some subset of $S$
depending on $s$.  Since each relator has even length as a word in the
elements of $S$, one may define a group homomorphism from $W$ to the
cyclic group $\Z/2$ by $w\mapsto [l(w)]\in \Z/2$, where $l(w)\in \Z$
denotes the length of $w$ as a word in $S$, and $[l(w)]$ its image in
$\Z/2$.  The kernel of this homomorphism will be denoted $W^\ev$, and
consists of the elements of $W$ that are expressible as words of even
length in the elements of $S$.  The right angled Coxeter system
$(W,S)$ is determined by the graph $L^1(W,S)$ with vertex set $S$ and
edges those pairs of vertices that commute.  Equivalently, the
right-angled Coxeter system is determined by the flag complex $L(W,S)$
with vertex set $S$ and simplices the cliques in the graph $L^1(W,S)$.

Given a right-angled Coxeter system $(W, S)$, the {\it Davis complex}
$\Sigma(W, S)$ can be realized as either a cubical complex, or as a 
simplicial complex which is the barycentric subdivision of the cubical
complex.  The simplicial structure is easier to describe,
so we consider this first.  A {\it spherical} subset $T$ of $S$ is a
subset whose members all commute; equivalently $T$ is either the 
empty set, or a subset of $S$ that spans a simplex of $L(W,S)$.  
A {\sl special parabolic} subgroup of $W$ is the subgroup of $W$
generated by a spherical subset $T$.  We denote the special 
parabolic subgroup generated by $T$ by $W_T$.  A {\sl parabolic}
subgroup of $W$ is a conjugate of a special parabolic subgroup.  The set 
of cosets of all special parabolic subgroups forms a poset, ordered by 
inclusion, and the simplicial complex $\Sigma(W,S)$ is the 
realization of this poset.  By construction, $W$ acts admissibly
simplicially on $\Sigma(W,S)$ in such a way that each stabilizer
subgroup is parabolic.  

If $T$ is a spherical subset of $S$, then
the subposet of cosets contained in $W_T$ is
equivariantly isomorphic to the poset of faces of the standard
$|T|$-cube $[-1,1]^T$, with the group $W_T \cong C_2^T$
acting via reflections in the coordinate hyperplanes.   
In this way we obtain a cubical structure on $\Sigma(W,S)$, 
in which the $n$-dimensional subcubes correspond to cosets 
$wW_T$ with $|T|=n$.  The setwise stabilizer of the cube 
$wW_T$ is the parabolic subgroup $wW_Tw^{-1}$, which acts
on the cube in such a way that the natural generators 
$wtw^{-1}$ act as reflections in the coordinate hyperplanes.  
The simplicial complex described above is the barycentric 
subdivision of this cubical complex.  

If we view every simplicial complex as containing a unique 
$-1$-simplex corresponding to the empty subset of its vertex set, 
then we get a natural bijective correspondence between
the $W$-orbits of cubes in $\Sigma(W,S)$ and the 
simplices of $L(W,S)$ which preserves incidence (the 
empty simplex corresponds to the 0-cubes).  
Hence we obtain: 

\begin{proposition}\label{natbij} 
There is a natural bijection between subcomplexes of the 
simplicial complex $L(W,S)$ and non-empty $W$-invariant 
subcomplexes of the cubical complex $\Sigma(W,S)$.  
\end{proposition} 

To show that $\Sigma(W,S)$ is a model for $\ebar W$, metric 
techniques are helpful.  There is a natural CAT(0)-metric on 
$\Sigma(W,S)$, which is best understood in terms of the cubical 
structure.  The length of a
piecewise linear path in $\Sigma(W,S)$ is defined using the standard
Euclidean metric on each cube, and the distance between two points of
$\Sigma(W,S)$ is the infimum of the lengths of PL-paths connecting 
them.  According to Gromov's criterion~\cite{gromov}, $\Sigma(W,S)$ 
is locally CAT(0) because the link of every vertex is isomorphic to 
$L(W,S)$ which is a flag complex (see \cite{davis2, gromov}).  It is 
easy to see that $\Sigma(W,S)$ is simply connected (for example, because 
its 2-skeleton is a version of the Cayley complex for $W$), and 
it follows that $\Sigma(W,S)$ is CAT(0)~\cite[Theorem II.4.1]{BridHaef}.  
Given that $W$ acts isometrically with finite stabilizers on $\Sigma(W,S)$ 
it follows that $\Sigma(W,S)$ is a model for $\ebar W$ via the
Bruhat-Tits fixed point theorem~\cite[p.~179]{BridHaef}~or~\cite[Prop.~3]{BLN}.

\begin{lemma}\label{conjugacy} Every finite subgroup of $W$ is a
  subgroup of a parabolic subgroup of $W$. In particular,
  there are finitely many conjugacy classes of finite subgroups of $W$
  and every finite subgroup is isomorphic to a direct product
  $(\Z/2)^k$ for some $0\leq k \leq n$ where $n$ is the dimension of
  $\Sigma(W, S)$.  
\end{lemma}

\begin{proof} Let $F$ be a finite subgroup of $W$.  By the Bruhat-Tits
  fixed point theorem $F$ fixes some point of $\Sigma$, and hence $F$
  is a subgroup of a point stabilizer.  Every such subgroup is
  parabolic, and each is conjugate to one of the finitely many
  special parabolics.   
\end{proof} 

Recall that a group is said to be of {\it type $F$} if it admits a
compact classifying space. 

\begin{corollary} The commutator subgroup $W'$ of $W$ is a
  finite-index torsion-free subgroup of type $F$.    
\end{corollary} 

\begin{proof} 
The abelianization of $W$ is naturally isomorphic to $C_2^S$.  
Every parabolic subgroup of $W$ maps injectively into $C_2^S$.  
It follows that $W'$ acts freely on the finite-dimensional 
contractible space $\Sigma$.  Hence $\Sigma/W'$ is a compact 
$K(W',1)$, from which it follows that $W'$ is both type $F$ and 
torsion-free.  
\end{proof}

\begin{lemma}[\cite{BLN}]\label{modW} The quotient of the pair
  $(\Sigma, \Sigma^\sing)$ by $W$ is isomorphic to the pair
  $(CL', L')$, i.e., the pair consisting of the cone on the 
barycentric subdivision of $L$ and its base.  This isomorphism is 
natural for automorphisms of $L$.  If $L$ is acyclic 
then so is $\Sigma^\sing$.  If $L$ is simply-connected, then 
so is $\Sigma^\sing$.  
\end{lemma}

\begin{proof} 
The first part is clear from the simplicial description of $\Sigma$.  
Now let $V$ be the unique free $W$-orbit of vertices in the simplicial
description of $\Sigma$.  The star of each $v\in V$ is a copy 
of the cone $CL'$, with $v$ as its apex.  The subcomplex of 
$\Sigma$ consisting of all simplices not containing any vertex 
of $V$ is $\Sigma^\sing$.  Hence $\Sigma$ is obtained from 
$\Sigma^\sing$ by attaching cones to countably many subcomplexes
isomorphic to $L'$.  

In the case when $L$ is acyclic, attaching a cone to a copy 
of $L'$ does not change homology.  It follows that $\Sigma^\sing$ 
must be acyclic since $\Sigma$ is.  Similarly, if $L$ is 
simply-connected, then attaching a cone to a copy of $L'$ does
not change the fundamental group, so $\Sigma^\sing$ must be 
simply-connected since $\Sigma$ is.  
\end{proof}

Now suppose a finite group $Q$ acts by automorphisms on $L(W,S)$.  
This defines an action of $Q$ on $W$, and hence a semidirect product 
$G=W\semi Q$.  

\begin{lemma} \label{equivariant} 
There is an admissible simplicial $G$-action on $\Sigma(W,S)$ 
extending the action of $W$, and $\Sigma(W,S)$ becomes a 
cocompact model for $\ebar G$.  
\end{lemma} 

\begin{proof} The action of $Q$ on the poset underlying $\Sigma(W,S)$ 
is defined in such a way that $q\in Q$  
sends the coset $wW_T$ to the coset $q(w)W_{q(T)}$.  This combines 
with the $W$-action to give an admissible $G$-action on $\Sigma(W,S)$.  
Since $\Sigma(W,S)$ is CAT(0) and the stabilizers are finite it 
follows that $\Sigma(W,S)$ is a model for $\ebar G$. 
\end{proof} 

\begin{corollary}\label{corequiv} 
Any finite-index subgroup $H$ of $G$ as above admits a cocompact 
model for $\ebar H$ and is virtually torsion-free. 
\end{corollary} 

\begin{remark}\label{remstabs} 
For the action of $G$ on $\Sigma$, the stabilizer of the vertex $W_T$ 
is the semidirect product $W_T\semi Q_T$, where $Q_T:=\{q\in Q: q(T)=T\}$.  
If $Q$ acts admissibly on $L$ then $Q_T$ fixes each element of $T$ and
the stabilizer is the direct product $W_T\times Q_T$.  Similarly, the 
stabilizer of the vertex $wW_T$ is the direct product
$wW_Tw^{-1}\times wQ_Tw^{-1}$.  Note in particular that the image 
of the stabilizer under the quotient map $G\to G/W\cong Q$ depends only 
on $T$, and not on $w$.  
\end{remark}

\begin{lemma} \label{mainlem}
Let $N$ be a finite index normal subgroup of $W^\ev$.  There is an
isomorphism $\psi$ from the relative chain complex $C_*(CL',L')$ to a
direct summand of the simplicial chain complex $C_*(\Sigma/N)$.  This
isomorphism is natural for automorphisms of $L$ that preserve $N$.  It
is also natural for the inclusion of subcomplexes in $L$ and the
corresponding $W/N$-invariant subcomplexes of the cubical structure on
$\Sigma/N$.
\end{lemma} 

\begin{proof} 
The cone $CL'$ is the realization of the poset of spherical 
subsets of $S$, with cone point the empty set $\emptyset$.  
For $\sigma$ a simplex of $CL'$, $\psi(\sigma)$ in  
$C_*(\Sigma/N)$ will be the signed sum of its $|W/N|$ 
inverse images under the map $\Sigma/N\rightarrow 
\Sigma/W=CL'$.  The signs will ensure that simplices 
of $CL'$ that do not contain $\emptyset$ as a vertex map 
to zero.  

In more detail, fix a transversal $w_1,\ldots,w_m$ to 
$N$ in $W$, and for $\sigma$ a simplex of $CL'$, viewed 
as a chain $\sigma=(T_0<T_1<\cdots<T_r)$ of spherical subsets, 
define 
\[\psi(\sigma) = \sum_{i=1}^m (-1)^{l(w_i)} w_i\sigma 
= \sum_{i=1}^m (-1)^{l(w_i)} (w_iW_{T_0}<\cdots<w_iW_{T_r}).\] 
Here $l(w)$ denotes the length of $w$ as a word in $S$.  For 
any $n\in N$, $l(wn)-l(w)$ is even, and so the sum above 
does not depend on the choice of transversal.  The above 
formula clearly describes a chain map from $C_*(CL')$ to 
$C_*(\Sigma/N)$.  Now if $T$ is a non-empty spherical subset
of $S$, $W_T$ contains equal numbers of words of odd and 
even length, and hence so does its image $W_T/(W_T\cap N)\leq W/N$.
Equivalently, any transversal to $W_T\cap N$ in $W_T$ contains equal
numbers of words of odd and even length.  
It follows that if $T_0\neq \emptyset$, then $\psi(\sigma)=0$. 
Hence the formula given above defines a chain map 
$\psi: C_*(CL',L')\rightarrow C_*(\Sigma/N)$.  This 
clearly has the claimed naturality properties.  

It remains to exhibit a splitting map $\phi: C_*(\Sigma/N)\rightarrow
C_*(CL',L')$.  This uses a `simplicial excision map'.  Let $v$ be the
image of $W_\emptyset\in\Sigma$ in $\Sigma/N$, and let $X$ be the
subcomplex of $\Sigma/N$ consisting of all simplices that do not have
$v$ as a vertex.  There is a natural bijection between simplices of
$CL'$ containing the cone vertex and simplices of $\Sigma/N$
containing $v$.  This induces an isomorphism $C_*(\Sigma/N,X) \cong
C_*(CL',L')$ and $\phi$ is defined as the composite of this with the
map $C_*(\Sigma/N)\rightarrow C_*(\Sigma/N,X)$.  To check that
$\phi\circ\psi$ is the identity map on $C_*(CL',L')$, let
$\sigma=(T_0<T_1<\cdots<T_r)$ be any $r$-simplex of $CL'$.  If
$T_0\neq \emptyset$ then we already know that
$\phi\circ\psi(\sigma)=\phi(0)=0$, and so the given formula for
$\phi\circ\psi$ does define a self-map of $C_*(CL',L')$.  On the other
hand, if $T_0=\emptyset$ then $\psi(\sigma)$ contains $m=|W:N|$
distinct signed simplices, exactly one of which has $v=W_\emptyset$ as
a vertex rather than some other coset of $W_\emptyset$; furthermore
this simplex appears with sign $+1$.  It follows that in this case
$\phi\circ\psi(\sigma)=\sigma$, confirming that $\phi\circ\psi$ is
the identity map of $C_*(CL',L')$.  
\end{proof} 

\begin{corollary}\label{cormainlem} 
With notation as above, let $K$ be a subcomplex of $L$, and 
let $\Sigma(K)$ be the (barycentric subdivision of the)  
cubical $W_L$-subcomplex of $\Sigma$ associated to $K$.  There 
is a natural isomorphism $\psi$ from the relative chain complex 
$C_*(CL',L'\cup CK')$ to a direct summand of the relative
simplicial chain complex $C_*(\Sigma/N,\Sigma(K)/N)$.  
\end{corollary} 

\begin{proof} 
$C_*(CK',K')$ is a subcomplex of $C_*(CL',L')$ and the 
corresponding quotient is $C_*(CL',L'\cup CK')$.  Similarly, 
$C_*(\Sigma(K)/N)$ is a subcomplex of $C_*(\Sigma/N)$ with 
$C_*(\Sigma/N,\Sigma(K)/N)$ the corresponding quotient.

By naturality of $\psi$ and $\phi$ we get a diagram as follows, 
in which the two left-hand squares with the same label on both 
vertical sides commute and such that the two composites labelled 
$\phi\circ\psi$ are equal to the relevant identity maps.  A
diagram chase shows that there are unique maps $\psi$~and~$\phi$ 
corresponding to the dotted vertical arrows that make the right-hand
squares with the same label on both 
vertical sides commute, and that these maps also satisfy $\phi\circ\psi=1$.  
\[\xymatrix{
0\ar[r]&C_*(CK,K')\ar@/^/[d]^\psi \ar[r]& C_*(CL',L')\ar@/^/[d]^\psi 
\ar[r]&C_*(CL',L'\cup CK')\ar@{-->}@/^/[d]^\psi\ar[r]& 0\\
0\ar[r]&C_*(\Sigma(K)/N)\ar@/^/[u]^\phi \ar[r]& C_*(\Sigma/N)\ar@/^/[u]^\phi 
\ar[r]&C_*(\Sigma/N,\Sigma(K)/N)\ar@{-->}@/^/[u]^\phi\ar[r]&0\\
}
\] 
\end{proof}

\section{Proof of Theorem~\ref{th:ksbq}} 

As in the statement of Theorem~\ref{th:ksbq}, let $L$ be a finite
$n$-dimensional flag complex equipped with an admissible simplicial
action of a finite group $Q$, let $(W,S)=(W_L,S_L)$ be the associated
right-angled Coxeter system, let $N$ be a finite-index subgroup of $W$
that is normalized by $Q$, and let $G$ be the semidirect product 
$G=N\semi Q$.  

\begin{proof}[Proof of Theorem~\ref{th:ksbq}]
Note that $G$ may be viewed as a finite-index 
subgroup of the semidirect product $W\semi Q$.  Under these
hypotheses, we already see from Corollary~\ref{equivariant} that 
the Davis complex $\Sigma$ is a cocompact $(n+1)$-dimensional 
model for $\ebar G$, and that $G$ is virtually torsion-free.  

Using the hypothesis that $L$ is acyclic, we see that the subcomplex
$\Sigma^{\sing(W)}$ of $\Sigma$ consisting of those points whose
stabilizer in $W$ is non-trivial is acyclic by Lemma~\ref{modW}.  In
this case any finite-index torsion-free subgroup of $G$ acts freely on
the acyclic $n$-dimensional complex $\Sigma^{\sing(W)}$, which 
implies that $\vcd G\leq n$.  

For the remainder of the proof it will be convenient to define 
$K:=L^\sing$, the subcomplex of $Q$-singular points in $L$.  
(We warn the reader that our use of `$K$' is different
to that in~\cite[ch.~7--8]{davis2}.)  

Next we show that $\cdbar G\geq n+1$.  Since $W^\ev$ has index~2
in $W$ and is clearly $Q$-invariant, we see that $(N\cap W^\ev)\semi
Q$ is a subgroup of $G$ of index at most~2.  Hence without loss of
generality we may assume that $N\leq W^\ev$.  Now consider the family
$\mW$ of finite subgroups of $G$, consisting of those finite subgroups
that are contained in $N$, or equivalently the finite subgroups that
map to the trivial subgroup under the factor map $G\rightarrow Q$.
The stabilizers in $W\semi Q$ of vertices of $\Sigma$ are described in
Remark~\ref{remstabs}, and by intersecting with $G=N\semi Q$ we get a
similar description of stabilizers in $G$: the stabilizer of the
vertex $wW_T$ is the direct product of the intersection $N\cap
wW_Tw^{-1}$ and a subgroup that maps isomorphically to $Q_T$, the
stabilizer in $Q$ of the vertex $T$ of $CL'$.  It follows that
$\Sigma^{\mW\mhyph\sing}$ is equal to the inverse image in $\Sigma$ of
the $Q$-singular set $CK'$ in $\Sigma/W=CL'$.  Hence
$\Sigma^{\mW\mhyph\sing}$ is the $W$-invariant subcomplex of the
cubical structure on $\Sigma$ that corresponds (under the map of
Proposition~\ref{natbij}) to $K=L^\sing$.  Using
Corollary~\ref{cormainlem} applied in this case, we see that
$H^{n+1}(\Sigma/N,\Sigma^{\mW\mhyph\sing}/N)$ admits a split
surjection onto $H^{n+1}(CL',L'\cup CK')$, which is isomorphic to
$H^n(L,K)=H^n(L,L^\sing)$ by excision.  Proposition \ref{bredon_prop}
finishes the argument.

To show that $\vcd G=n$ when $L$ is a barycentric subdivision, we use
the calculation of the cohomology of $W$ with free coefficients as
described in~\cite[section~8.5]{davis2}.  If $v$ is a vertex of $L$
that corresponds to the barycentre of an $n$-dimensional cell, then
$L-v$ is homotopy equivalent to the subcomplex obtained from $L$ by
removing the (interior of the) $n$-dimensional cell.  Hence we see that 
$H^{n-1}(L-v)\cong \Z$, and so by~\cite[cor.~8.5.3]{davis2}, $H^n(W;\Z
W)$ contains a free abelian summand.

Now we show that $\vcd G=n$ in the case when $K=L^\sing$ is a full
subcomplex of $L$.  From the long exact sequence for the pair $(L,K)$
we see that $H^{n-1}(K)\neq 0$, and hence $H^n(CK,K)\neq 0$.
Lemma~\ref{mainlem} applied to the Coxeter group $W_K$ and its
finite-index torsion-free subgroup $W'_K$ shows that $H^n(W'_K;\Z)=
H^n(\Sigma_K/W'_K)$ contains a summand isomorphic to $H^n(CK,K)$ and
so is not zero.  (Here we use $\Sigma_K$ to denote $\Sigma(W_K,S_K)$
since we reserve $\Sigma$ to stand for $\Sigma(W_L,S_L)$.)  Since $K$
is a full subcomplex of $L$, the Coxeter group $W_K$ is naturally a
subgroup of $W_L\leq G$, and hence $\vcd G\geq \vcd W_K \geq n$.

The general case of Theorem~\ref{th:ksbq} follows from the 
case described above, but since we will make extensive use
of the K\"unneth theorem it is helpful to work with cohomology
with coefficients in a finite field rather than integral cohomology.
Since $L$ is finite and $n$-dimensional, the hypothesis that
$H^n(L,L^\sing)\neq 0$ is equivalent to the existence of a prime~$p$
for which the mod-$p$ cohomology group $H^n(L,L^\sing;\F_p)\neq 0$.
Similarly, the hypothesis that
$\bigotimes_{i=1}^m H^{n_i}(L_i,L_i^\sing)\neq 0$
is equivalent to the existence of a single prime $p$ such that for
each $i$, $H^{n_i}(L_i,L_i^\sing;\F_p)\neq 0$.  For the remainder of
the proof, we fix such a prime.  The mod-$p$ analogues of
Proposition~\ref{bredon_prop} and Corollary~\ref{cormainlem} are
easily deduced from the integral versions.  

Now let each $G_i$ be defined as above 
in terms of $L_i$, $Q_i$, $n_i$ and $N_i\leq W_i$, and define 
$\Gamma:=G_1\times\cdots\times G_m$, $Q=Q_1\times\cdots\times Q_m$, 
and $W:=W_1\times \cdots\times W_m$.  Finally, let $n:=\sum_{i=1}^m n_i$.  
The direct product $\Sigma:=\Sigma_1\times \cdots\times \Sigma_m$ is a 
cocompact model for $\ebar\Gamma$ of dimension $m+n$, and so 
$\cdbar \Gamma\leq m+n$.  Also 
the direct product $\Sigma_1^\sing\times\cdots\times \Sigma_m^\sing$ 
is an acyclic $n$-dimensional simplicial complex admitting a proper
$\Gamma$-action, which implies that $\vcd\Gamma\leq n$.  

The lower bounds also work just as in the case $m=1$; first we 
consider $\cdbar\Gamma$.  If we define
$\mW$ to be the family of finite subgroups of $\Gamma$ that are contained 
in $W$, then a point 
$x=(x_1,\ldots,x_m)\in \Sigma= \Sigma_1\times\cdots\times \Sigma_m$ 
is in $\Sigma^{\mW\mhyph\sing}$ if and only if there is an $i$ so 
that $x_i\in \Sigma_i^{\mW_i\mhyph\sing}$.  Hence we see that 
if we define $N:=N_1\times\cdots\times N_m$, then 
$$\Sigma^{\mW\mhyph\sing}=\bigcup_{i=1}^m\Sigma_1\times\cdots
\times\Sigma^{\mW_i\mhyph\sing}\times\cdots\times \Sigma_m,$$ 
and by the relative K\"unneth Formula
$H^{n+m}(\Sigma/N,\Sigma^{\mW\mhyph\sing}/N;\F_p)$ contains a direct
summand isomorphic to  
$\bigotimes_{i=1}^m
H^{n_i+1}(\Sigma_i/N_i,\Sigma_i^{\mW_i\mhyph\sing}/N_i;\F_p),$
which is non-zero since it contains a 
summand isomorphic to 
\[\bigotimes_{i=1}^m H^{n_i+1}(CL_i,L_i\cup CK_i;\F_p) \cong 
\bigotimes_{i=1}^m H^{n_i}(L_i,K_i;\F_p)= 
\bigotimes_{i=1}^m H^{n_i}(L_i,L_i^\sing;\F_p).\] 

To give a lower bound for $\vcd\Gamma$, start by considering the two
extra hypotheses separately for each $i$.  If $L_i$ is a barycentric
subdivision, then as above $H^{n_i}(W_i;\Z W_i)$ contains a free
abelian summand, and so by the universal coefficient theorem
$H^{n_i}(W_i;\F_p W_i)\neq 0$.  If instead $K_i:=L_i^\sing$ is a full
subcomplex of $L_i$, then $H^{n_i}(W_i';\F_p)\neq 0$ as above.  There
is a surjective homomorphism from $\F_pW'_i$ onto $\F_p$ and hence a
short exact sequence of $\F_pW'_i$ modules
\[0\rightarrow I\rightarrow
\F_pW'_i\rightarrow \F_p \rightarrow 0\] 
for suitable $I$.  The
corresponding long exact sequence in cohomology implies that
$H^{n_i}(W'_i;\F_pW'_i)\rightarrow H^{n_i}(W'_i;\F_p)$ is surjective,
since its cokernel is contained in $H^{n_i+1}(W'_i;I)=0$.  It follows
that $H^{n_i}(W_i;\F_p W_i)\cong H^{n_i}(W'_i;\F_pW'_i)\neq 0$.  Since
$W$ acts cocompactly on $\Sigma=\Sigma_1\times\cdots\times \Sigma_m$,
the universal coefficient theorem for cohomology with compact supports
may be applied~\cite[8.5.9]{davis2}.  Hence
\[H^n(W;\F_pW) \cong \bigotimes_{i=1}^m H^{n_i}(W_i;\F_pW_i) \neq 0,\] 
showing that $\vcd\Gamma\geq n$ as required.  
\end{proof}

\section{Examples}

In this section we construct sufficiently many examples of finite
groups $Q$ and $Q$-CW-complexes $L$ to establish
Corollary~\ref{cora}.  First we collect some results concerning
triangulations.  

\begin{proposition} Any finite $Q$-CW-complex is equivariantly
  homotopy equivalent to a finite simplicial complex of the same dimension
  with an admissible $Q$-action.  If $L$ is any simplicial complex
  with $Q$-action, the $Q$-action on the barycentric subdivision $L'$ of $L$ is
  admissible.  For any admissible action of $Q$ on
  $L$, $L^\sing$ is a subcomplex.  If $M\leq L$ is any subcomplex of a
  simplicial complex $L$, then its barycentric subdivion $M'$ is a
  full subcomplex of the flag complex $L'$.  
\end{proposition}

\begin{proof}
  The first claim follows easily from the simplicial approximation
  theorem.  Simplices of $L'$ correspond to chains in the poset of
  simplices of $L$; since $Q$ acts as automorphisms of this poset the
  action on $L'$ is admissible.  For an admissible action of $Q$ on $L$, a
  simplex of $L$ is fixed by $H\leq Q$ if and only if each of its
  vertices is fixed.  Hence each $L^H$ is the full subcomplex on the
  $H$-fixed vertices, and $L^\sing=\bigcup_{1<H\leq Q}L^H$ is a
  subcomplex.  Finally if $M$ is any subcomplex of $L$, the poset of
  simplices of $M$ is a subposet of the poset of simplices of $L$, and
  so $M'$ is a full subcomplex of $L'$.
\end{proof}

\noindent 
{\bf Example 1.} Let $Q$ be the alternating group $A_5$, and define a
$Q$-CW-complex as follows.  For the 1-skeleton $L^1$ of $L$ take the
complete graph on five vertices, with the natural action of $Q=A_5$.
In $A_5$, the 24 elements of order five split into two conjugacy
classes of size 12, and any element $g$ of order 5 is conjugate to
$g^{-1}$ (but is not conjugate to $g^2$ or $g^3$).  Define $L$ by
using one of the two conjugacy classes of 5-cycles to describe
attaching maps for six pentagonal 2-cells.  By construction there is a
$Q$-action without a global fixed point and it is easily checked that
$L$ is acyclic.  In fact, $\pi_1(L)$ is isomorphic to $SL(2,5)$, the
unique perfect group of order 120, and $L$ is isomorphic to the
2-skeleton of the Poincar\'e homology sphere~\cite[I.8]{Bredon}.  The
singular set for the $Q$-action consists of the 1-skeleton and the
five lines of symmetry of each pentagonal 2-cell.  Equivalently the
singular set is the 1-skeleton of the barycentric subdivision of $L$
(i.e., the simplicial complex with 21 vertices coming from the poset
of faces of $L$).  In particular $H^2(L,L^\sing)\neq 0$.  For this
21-vertex triangulation, $L^\sing$ is not a full subcomplex of $L$.
This could be rectified by taking a finer triangulation, but instead
note that $L$ is the barycentric subdivision of a polygonal complex.
By taking each $Q_i$ to be $A_5$ and each $L_i$ to be this 21-vertex
triangulation of $L$, we obtain groups $\Gamma_m$ having the
properties stated in the first part of Corollary~\ref{cora}.

\medskip\noindent
{\bf Example 2.}  Fix distinct primes $p$ and $q$, and let 
$Q$ be cyclic of order $q$, generated by $g$.  For the 
$Q$-fixed point set $L^Q$, take a mod-$p$ Moore space 
$M(1,p)$.  This space has a CW-structure with 1-skeleton
a circle and one 2-cell $f$.  The 2-cell $f$ is attached 
to the circle via a map of degree~$p$.  Now define $L^2$ 
by adding on a free $Q$-orbit of 2-cells $f_0,\ldots,f_{q-1}$, 
where $f_i=g^if_0$, so that each $f_i$ is attached to the 
circle by a degree one map.  $L^2$ is simply connected, 
and $H_2(L^2)$ is a free $\Z Q$-module of rank one, 
since it has a $\Z$-basis given by the elements 
$$e_j:=f - \sum_{i=0}^p g^i f_j = f-\sum_{i=0}^p g^{i+j}f_0$$
for $0\leq j<q$, and $g^je_0= e_j$ for each $j$.  Make 
$L$ by attaching a free $Q$-orbit of 3-cells to kill each 
$e_j$, so that $L$ is acyclic (and also contractible).
The long exact sequence for the pair $(L,L^Q)=(L,L^\sing)$
implies that $H^3(L,L^Q)\cong \Z/p$.  
To establish the second part of Corollary~\ref{cora}
we take each $Q_i$ to be cyclic of order $q_i$, take 
each $L_i$ to be a suitable triangulation of the above 
$Q_i$-CW-complex for some fixed choice of $p$, 
and take $N_i$ to be the commutator subgroup of the 
Coxeter group $W_i:=W_{L_i}$.  For any such choice, 
we obtain a group $\Lambda_m$ as in the statement.  
To ensure that $\Lambda_m$ contains only cyclic finite 
subgroups we must take the primes $q_i$ all distinct, 
whereas to ensure that $\Lambda_m$ contains only abelian 
finite subgroups of exponent $q$ we take $q_i=q$ for all~$i$.  

\section{Contractibility and acyclicity} 

In~\cite{BLN}, it was shown that certain right-angled Coxeter groups 
$W$ have the property that $\vcd W=\cdbar W =2<\gdbar W=3$.  In this section
we improve this result by showing that for these same groups there 
is no 2-dimensional contractible proper $W$-CW-complex.  

We will use a few subsidiary results in the proof.  Results similar to
Propositions \ref{propsubcx}~and~\ref{propacyc} appear
in~\cite{casadicks}, and with extra hypotheses in~\cite{segev}.
Proposition~\ref{kervair} is a corollary of the celebrated 
Gerstenhaber-Rothaus theorem~\cite{GR}.  

\begin{proposition} If $Y$ is a subcomplex of a 2-dimensional 
acyclic complex, then $H_2(Y)=0$ and each $H_i(Y)$ is 
free abelian. \label{propsubcx}
\end{proposition}

\begin{proof} 
If $Y$ is any subcomplex of an $n$-dimensional acyclic 
complex $Z$, then consideration of the homology long exact 
sequence for the pair $(Z,Y)$ shows that $H_n(Y)$ is 
trivial and that $H_{n-1}(Y)$ is free abelian.  Since 
$H_0$ is always free abelian, the case $n=2$ gives the 
claimed result.  
\end{proof}

\begin{proposition}\label{propacyc}
Let $Q$ be a finite soluble group and let $X$ be a 2-dimensional 
acyclic $Q$-CW-complex.  Then the fixed point set $X^Q$ is also acyclic. 
\end{proposition} 

\begin{proof} 
The finite soluble group $Q$ has a normal subgroup $N$ of prime index, 
the factor group $Q/N$ acts on the $N$-fixed point set $X^N$, and
the equality $X^Q=(X^N)^{Q/N}$ holds.  Hence it suffices to consider 
the case in which $Q$ has prime order.  

By the P.~A.~Smith theorem, $X^Q$ is mod-$p$ acyclic in the case 
when $Q$ has order $p$.  By the previous proposition, $H_i(X^Q)$ 
is free abelian for all $i$.  By 
the universal coefficient theorem, the rank of the $i$th mod-$p$ 
homology group of $X^Q$ is equal to the rank of $H_i(X^Q)$.  
Hence $X^Q$ must be acyclic.  
\end{proof}   

\begin{proposition}\label{kervair} Let $\Gamma$ be a group and
  $\rho:\Gamma\to U(n)$ be a unitary representation of $\Gamma$.  
  Define $\widetilde \G:=\Gamma \ast \langle x_1, \dots,
  x_r\rangle/{\langle\langle w_1, \dots, w_r\rangle\rangle}$ where
  each $w_i$ is a word in elements of $\G$ and $x_1, \dots, x_r$. Let
  $d_{ij}$ be the total exponent of $x_j$ in $w_i$ and set
  $d=\det(d_{ij})$. If $d\ne 0$, then $\rho$ extends to a representation
  $\tilde \rho: \widetilde \Gamma\to U(n)$.   
\end{proposition}

\begin{proof} Extending $\rho$ to a representation of $\widetilde
  \Gamma$ is equivalent to finding solutions $\overline{x}_i\in U(n)$
  to the system of equations $\overline{w}_1=\cdots=\overline{w}_r=1$.   
  Here $\overline{w}_i$ is the word in elements of $U(n)$ and
  variables $\overline{x}_1,\ldots,\overline{x}_r$ corresponding to
  the word $w_i$.  In more detail, the elements
  of $U(n)$ appearing in $\overline{w}_i$ are obtained by applying
  $\rho$ to the elements of $\Gamma$ appearing in the word $w_i$,
  while each occurrence of $x_i$ is replaced by $\overline{x}_i$.  
  When such a solution has been found, we may define
  $\widetilde{\rho}(x_i):= \overline{x}_i$.  The existence of a
  solution to this system is established in~\cite[theorem~1]{GR}.
  \end{proof} 

We recall that the {\sl nerve} of a covering is the simplicial complex whose
vertices are the sets in the cover, and whose simplices are the finite
collections with a non-empty intersection~\cite[section~3.3]{hatcher}.  

\begin{lemma} \label{lemmahelly}
Let $X$ be a CW-complex, let $S$ be a finite indexing 
set, and let $X(s)$ be a subcomplex of $X$ such that each 
$X(s)$ is acyclic and each intersection of $X(s)$'s is either 
empty or acyclic.  Define 
$$X^\#:= \bigcup_{s\in S} X(s),$$ 
and let $|\mathcal{N}|$ be the realization of the nerve of the 
covering of $X^\#$ by the subcomplexes $X(s)$.  There is a map 
$f:X^\#\rightarrow |\mathcal{N}|$ which is a homology isomorphism 
and induces a surjection of fundamental groups.  
\end{lemma}

\begin{proof} 
In the case when each intersection of $X(s)$'s is either contractible
or empty, it is well-known that there is a homotopy equivalence 
$f:X^\#\rightarrow |\mathcal{N}|$~\cite[4.G, Ex.~4]{hatcher}.  
We use Quillen's plus construction to reduce to this case.  

For $T\subseteq S$, 
define $X(T)$ to be the intersection $X(T)=\bigcap_{s\in T} X(s)$.  
Suppose that $U\subseteq S$ is such that 
$X(U)$ is non-empty.  In this case, since $X(U)$ is 
acyclic we can find a set $A_U$ of 2-cells with attaching 
maps from the boundary of the 2-cell to $X(U)$ so that 
each attaching map represents a conjugacy class of 
commutators in $\pi_1(X(U))$ and so that the fundamental 
group of the resulting complex $\widehat{X}(U)$ is trivial.  Moreover, there
is a set $B_U$ of 3-cells and attaching maps from the boundary 
of the 3-cell to $\widehat{X}(U)$ so that the resulting complex 
$X_U$ contains $X(U)$ as a subcomplex, is simply-connected, 
and such that the inclusion of $X(U)$ into $X_U$ is a homology 
isomorphism.  Define $Y$ by attaching to $X$ 2- and 3-cells 
indexed by $\coprod_{U}A_U$ and $\coprod_U B_U$ respectively. 
Define a subcomplex $Y(s)$ of $Y$ by attaching to $X(s)$ the 2- and
3-cells indexed by $\coprod_{s\in U} A_U$ and $\coprod_{s\in U} B_U$ 
respectively.  Finally define $Y(T):=\bigcap_{s\in T}Y(s)$, 
and $Y^\#:= \bigcup_{s\in S}Y(s)$.  The nerve of the covering 
of $Y^\#$ by the subcomplexes $Y(s)$ is naturally isomorphic to 
$\mathcal{N}$.  A Mayer-Vietoris spectral sequence argument 
shows that the inclusion $X^\#\rightarrow Y^\#$ is a homology 
isomorphism, and this map induces a surjection $\pi_1(X^\#)
\rightarrow \pi_1(Y^\#)$ because the 1-skeleta of $X^\#$ and 
$Y^\#$ are equal.  
\end{proof} 

It is also possible to prove the above result directly using the
Mayer-Vietoris spectral sequence and the van Kampen theorem to keep
track of the homology and fundamental group respectively.

\begin{proof}[Proof of Theorem~\ref{acycthm}]
The Davis complex $\Sigma=\Sigma(W_L,S_L)$ is a cocompact 
3-dimensional model for $\ebar W_L$.  Since $L$ is acyclic, $\Sigma^\sing$ 
is a 2-dimensional acyclic proper $W_L$-CW-complex in which the fixed 
point set for any non-trivial finite subgroup is contractible.  This 
suffices to show that $\cdbar W_L=2$.  

Now suppose that $X$ is any contractible proper 2-dimensional 
$W_L$-CW-complex.  Let $S=S_L$, and define $X^\#$ as the union 
of the fixed point sets $X^s$: $X^\#:=\bigcup_{s\in S}X^s$.  
By construction, the realization of the nerve of the covering of $X^\#$ by the 
sets $X^s$ is equal to $L$.  By Propostion~\ref{propacyc}, 
for each $T\subseteq S$ that spans a simplex of $L$ the subset 
$$X(T):=\bigcap_{s\in T}X^s= X^{\langle T\rangle}$$ 
is acyclic, and for each $T$ that does not span a simplex of 
$L$, $X(T)$ is empty.  By Lemma~\ref{lemmahelly}, it follows 
that $X^\#$ is acyclic and that there is a natural surjection 
$\phi:\pi_1(X^\#)\rightarrow \pi_1(L)$.  
Define $\rho':=\rho\circ \phi:\pi_1(X^\#)\rightarrow U(n)$, a
non-trivial unitary representation of $\pi_1(X^\#)$.  
We use this representation to obtain a contradiction.  

Pick $g\in \pi_1(X^\#)$ so that $\rho'(g)\neq 1$.  
Since $X$ is contractible, there exists a connected subcomplex
$X_1$ of $X$ with $X^\# \subseteq X_1$ such that $X_1-X^\#$ comprises 
only finitely many cells, and such that $g$ maps to the identity
element of $\pi_1(X_1)$.  By Proposition~\ref{propsubcx}, 
$H_2(X_1)=0$, and $H_1(X_1)$ is free abelian.  In general $X_1-X^\#$ will
contain some 0-cells; by contracting some of the 1-cells in $X_1-X^\#$
we may get rid of these extra 0-cells without changing the homotopy
type.  In this way we replace $X_1$ by a complex $X_2$ with the 
following properties: $X^\#\subseteq X_2$; $H_1(X_2)$ is free abelian 
and $H_2(X_2)=\{0\}$; $X_2$ consists of $X^\#$ with finitely many 
1- and 2-cells added; $g$ is in the kernel of the map 
$\pi_1(X^\#)\rightarrow \pi_1(X_2)$.  Unlike $X_1$, $X_2$ is not a
subcomplex of $X$ but this is irrelevant.  Since $X_2$ is made by
attaching finitely many cells to the acyclic complex $X^\#$, note that
$H_1(X_2)$ is free abelian of finite rank.  Now make $X_3$ by attaching 
2-cells to exactly kill $H_1(X_2)$.  Thus $X_3$ is an acyclic 
2-complex, obtained by attaching the same finite number, $r$ say,
of 1- and 2-cells to $X^\#$.  If we write $\Gamma=\pi_1(X^\#)$ and
$\widetilde\Gamma:=\pi_1(X_3)$, then the relationship between $\Gamma$
and $\widetilde\Gamma$ is exactly as in the hypotheses of
Proposition~\ref{kervair}.  Here, the group generator
$x_i\in\widetilde\Gamma$ corresponds to a based loop in $X_3$ that
remains in $X^\#$ except that it travels once along the $i$th of
the $r$ new 1-cells, and the word $w_i$ spells out the attaching map
for the $i$th of the $r$ new 2-cells as a word in the elements of $\Gamma$
and the new loops $x_j$.  Moreover, since $X^\#$ and $X_3$ are both
acyclic, the relative homology groups $H_i(X_3,X^\#)$ all vanish,
which tells us that the determinant $d$ appearing in the statement of
Proposition~\ref{kervair} is equal to $\pm 1$.  Now 
Proposition \ref{kervair} can be applied and tells us that the 
representation $\rho':\pi_1(X^\#)\rightarrow U(n)$ extends to 
a representation $\tilde \rho:\pi_1(X_3)\rightarrow U(n)$. 
However, this contradicts the fact that $\rho'(g)\neq 1$, 
while $g$ maps to the identity in $\pi_1(X_3)$.  
\end{proof} 

\begin{remark} 
As an example of a suitable $L$, take a flag triangulation 
of the 2-skeleton of the Poincar\'e homology sphere (which 
was discussed in the previous section); here 
there is a faithful representation 
$\rho:\pi_1(L)\cong SL(2,5)\rightarrow U(2)$.   
\end{remark} 

\begin{remark} 
There is a version of Brown's question that remains open: 
for $m>2$, is there a virtually torsion-free group $G$ 
such that $\vcd G= m$ but there exists no contractible 
$m$-dimensional proper $G$-CW-complex?   
\end{remark}

\leftline{\bf Authors' address:}

\obeylines

\smallskip
{\tt i.j.leary@soton.ac.uk\qquad n.petrosyan@soton.ac.uk} 

\smallskip
School of Mathematical Sciences, 
University of Southampton, 
Southampton,
SO17 1BJ


\begin{thebibliography}{19}
\bibitem{BB} Bestvina, M. and Brady, N., {\em Morse theory and finiteness properties of groups}, Invent. Math. {\bf 129} (1997) 445--470.  
\bibitem{BLN} Brady, N., Leary, I.~J., Nucinkis, B.~E,~A.,
    {\em On algebraic and geometric dimensions for group with
    torsion}, J.~London.~Math.~Soc., (2), {\bf 64}, (2001) 489--500. 
\bibitem{Bredon} Bredon, G. E.,
    {\em Equivariant cohomology theories}, Lecture Notes in
    Mathematics {\bf 34}, Springer (1967). 
\bibitem{Bredon2} Bredon, G. E.,
    {\em Introduction to compact transformation groups}, Pure and
    Applied Mathematics, {\bf 46}, Academic Press, New York and
    London, (1972). 
\bibitem{BridHaef} Bridson, M.R. and Haefliger, A.,
    {\em Metric spaces of non-positive curvature} Springer Verlag
    Vol. 319 (1999). 
\bibitem{brownwall} Brown, K.~S., {\em Groups of virtually finite 
dimension} in Homological Group Theory (Proc. Sympos. Durham 1977)
pp. 27--70, London Math. Soc. Lecture Note Ser. {\bf 36}, Cambridge 
Univ. Press (1979).  
\bibitem{brownco} Brown, K.~S., {\em Cohomology of Groups}, Graduate
  Texts in Math. {\bf 87}, Springer Verlag, New York (1982).
\bibitem{brown} Brown, K.~S., {\em Buildings}, Springer Verlag (1989).
\bibitem{casadicks} Casacuberta, C. and Dicks, W., {\em On finite
  groups acting on acyclic complexes of dimension two},
  Publ. Mat. {\bf 36} (1992) 463--466.  
\bibitem{davis1} Davis, M.~W.,
    {\em Groups generated by reflections}, Ann. Math. {\bf 117} (1983)
    293--324. 
\bibitem{davis2} Davis, M.~W.,
    {\em The Geometry and Topology of Coxeter Groups}, LMS Monographs
    Series Vol. {\bf 32}, Princeton (2007). 
\bibitem{DDMP} Degrijse, D. and Mart\'\i{n}ez-P\'erez, C., 
{\em Dimension invariants for groups admitting a cocompact model for 
proper actions}, J. reine und angew. Math., online Ahead of Print, 
DOI: 10.1515/crelle-2014-0061.   
\bibitem{DDNP} Degrijse, D. and Petrosyan, N., {\em Geometric
  dimension of groups for the family of virtually cyclic subgroups},
  J. Topol. {\bf 7} (2014) 697--726.   
\bibitem{dran} Dranishnikiov, A.~N., {\em Boundaries of Coxeter groups
  and simplicial complexes with given links}, J. Pure Appl. Algebra
  {\bf 137} (1999) 139--151.   
\bibitem{GR} Gerstenhaber, M. and Rothaus, O.~S., {\em The solution of
  sets of equations in groups}, Proc. Nat. Acad. Sci. U.S.A. {\bf 48}
  (1962) 1531--1533.   
\bibitem{gromov} Gromov, M.,
    {\em Hyperbolic groups}, Essays in group theory
    (ed. S. M. Gersten), Mathematical Sciences Research Institute
    Publications 8 (Springer, 1987) 75--264. 
\bibitem{hatcher} Hatcher, A., {\em Algebraic Topology}, Cambridge
  University Press (2002).   
\bibitem{VFG} Leary, I.~J. and Nucinkis, B.~E.~A., {\em Some groups of
  type $VF$}, Invent. Math. {\bf 151}, (2001) 135--165.  
\bibitem{Luck} L\"{u}ck, W.,
    {\em Transformation groups and algebraic K-theory}, Lecture Notes
    in Mathematics, Vol. {\bf 1408}, Springer-Berlin (1989). 
\bibitem{Luck1}  L\"uck, W.,
    {\em The type of the classifying space for a family of subgroups},
    J. Pure Appl. Algebra {\bf 149} (2000), 177--203. 
\bibitem{Luck2} L\"uck, W.,
    {\em Survey on classifying spaces for families of subgroups},
    Infinite Groups: Geometric, Combinatorial and Dynamical Aspects,
    Springer (2005), 269--322. 
\bibitem{Luck3} L\"uck, W.,
    {\em On the classifying space of the family of finite and of
    virtually cyclic subgroups for CAT(0)-groups}, M\"{u}nster J. of
    Math. {\bf 2} (2009), 201--214. 
\bibitem{LuckMeintrup} L\"{u}ck, W., and Meintrup, D.,
    {\em On the universal space for group actions with compact isotropy},
    Proceedings of the conference ``Geometry and Topology'' in Aarhus,
    (1998), 293--305. 
\bibitem{CMP} Mart\'\i{n}ez-P\'erez, C., {\em Euler classes and Bredon
  cohomology for groups with restricted families of finite torsion},
  Math. Z. {\bf 2013} 
  (2013) 761--780. 
\bibitem{prsw} Przytycki, P. and \'Swi\c{a}tkowski, J., {\em Flag-no-square 
triangulations and Gromov boundaries in dimension 3}, Groups
  Geom. Dyn. {\bf 3} (2009) 453--468.
\bibitem{segev} Segev, Y., {\em Group actions on finite acyclic
  simplicial complexes}, Israel J. Math. {\bf 82} (1993), 381--394.  
\end{thebibliography}
\end{document}